\documentclass{amsart}
\usepackage{amssymb,amsmath,amsrefs}
\usepackage[colorlinks=true]{hyperref}
\hypersetup{urlcolor=blue, citecolor=green}
\usepackage[dvips]{graphicx}
\usepackage{color}

\theoremstyle{plain}

\newtheorem{mainthm}{Theorem}

\newtheorem{mainclly}[mainthm]{Corollary}
\newtheorem*{conj*}{Conjecture}
\newtheorem*{cor*}{Corollary}
\newtheorem{theorem}{Theorem}[section]

\newtheorem{prop}[theorem]{Proposition}
\newtheorem{proposition}[theorem]{Proposition}
\newtheorem{corollary}[theorem]{Corollary}
\newtheorem{lemma}[theorem]{Lemma}
\newtheorem{lema}[theorem]{Lemma}

\newtheorem{question}{Question}
\newtheorem*{q*}{Question}

\theoremstyle{definition}
\newtheorem{definition}[theorem]{Definition}
\newtheorem*{def*}{Definition}
\newtheorem{remark}[theorem]{Remark}
\newtheorem{rmk}[theorem]{Remark}

\renewcommand{\epsilon}{\varepsilon}

\newcommand{\Z}{\mathbb{Z}}
\newcommand{\N}{\mathbb{N}}

\newcommand{\eps}{\varepsilon}

\newcommand{\dist}{\operatorname{\textit{d}}}

\newcommand{\diam}{\operatorname{diam}}

\title[Beyond topological hyperbolicity]{Beyond topological hyperbolicity:\\ the L-shadowing property}
\author[A. Artigue, B. Carvalho, W. Cordeiro and J. Vieitez]{Alfonso Artigue, Bernardo Carvalho, Welington Cordeiro \\ and Jos\'e Vieitez}
\date{\today}


\begin{document}
\begin{abstract}
In this paper we further explore the L-shadowing property defined in \cite{CC2}
for dynamical systems on compact spaces.
We prove that
structurally stable diffeomorphisms and some pseudo-Anosov diffeomorphisms of the two-dimensional sphere
satisfy this property. Homeomorphisms satisfying the L-shadowing property have a spectral decomposition where
the basic sets are either expansive or contain arbitrarily small topological semi-horseshoes (periodic sets where the restriction is semiconjugate to a shift).
To this end, we characterize the L-shadowing property using local stable and unstable sets and the classical shadowing property.
We exhibit homeomorphisms with the L-shadowing property and arbitrarily small topological semi-horseshoes without periodic points.
At the end, we show that positive finite-expansivity jointly with the shadowing property imply that the space is finite.
\end{abstract}

\maketitle

\section{Introduction and Statement of Results}

Hyperbolicity is one of the most important concepts in the theory of chaotic dynamical systems.
Since the seminal works of Anosov \cite{Anosov} and
 Smale \cite{Smale} it has been a main topic of research among
many mathematicians. In hyperbolic systems, each tangent space splits into two invariant subspaces, the first being uniformly contracted, and the second uniformly expanded, by the action of the derivative map. The dynamics of such systems can be well described in both topological and statistical viewpoints, so many effort is being made to understand the dynamics \emph{beyond uniform hyperbolicity}. The theory evolved into a lot of distinct directions and different generalizations of hyperbolicity can be found in the literature (see \cite{BDV}). If one is interested in considering dynamics from a topological point of view or to understand the dynamical behaviour of homeomorphisms, then it is natural to consider a topological concept for hyperbolicity. This has been done before in the literature by considering expansive homeomorphisms satisfying the shadowing property
defined on compact metric spaces (see \cite{AH}). These systems are usualy called \emph{topologically hyperbolic} since their dynamics and the uniformly hyperbolic dynamics are pretty much the same. For instance, they admit a local product structure, spectral decomposition, density of periodic points in the non-wandering set, limit shadowing, among others (see \cites{AH,P1}).

Examples of topologically hyperbolic diffeomorphisms of the torus $\mathbb{T}^2$ admitting non-hyperbolic fixed points can be found in \cite{Lewowicz}. These examples are on the boundary of the hyperbolic systems and are, indeed, topologically conjugate to Anosov diffeomorphisms of the torus. Hiraide proved in \cite{Hiraide} that all topologically hyperbolic homeomorphisms of the torus $\mathbb{T}^n$ must be topologically conjugate to Anosov diffeomorphisms of $\mathbb{T}^n$. This seems to indicate that topological hyperbolicity is not that far from uniform hyperbolicity when we look at the dynamics from a topological point of view. Thus, it is natural to consider generalizations of this topological notion of hyperbolicity. In this paper we explore the world \emph{beyond topological hyperbolicity} through the eyes of a dynamical property called L-shadowing, which is now defined.

\begin{def*}
We say that a homeomorphism $f$, defined in a compact metric space $(X,d)$, has the \emph{L-shadowing property}, if for every $\eps>0$, there exists $\delta>0$ such that for every sequence $(x_k)_{k\in\Z}\subset X$ satisfying $$d(f(x_k),x_{k+1})\leq\delta\,\,\,\,\,\, \text{for every} \,\,\,\,\,\, k\in\Z \,\,\,\,\,\, \text{and}$$  $$d(f(x_k),x_{k+1})\to0 \,\,\,\,\,\, \text{when} \,\,\,\,\,\, |k|\to\infty,$$ there is $z\in X$ satisfying $$d(f^k(z),x_k)\leq\eps \,\,\,\,\,\, \text{for every} \,\,\,\,\,\, k\in\Z \,\,\,\,\,\, \text{and}$$ $$d(f^k(z),x_k)\to0\,\,\,\,\,\, \text{when} \,\,\,\,\,\, |k|\to\infty.$$ In this case, we say that $(x_k)_{k\in\Z}$ is a $(\delta,L)$-\emph{pseudo orbit} of $f$ and that $(x_k)_{k\in\Z}$ is $(\eps,L)$-\emph{shadowed} by $z$.
\end{def*}

The L-shadowing property has similarities with the classical notions of shadowing and limit-shadowing, but we enlighten here an important difference which plays a significant role in our paper: the existence of one point that satisfies both the $\eps$-shadowing and limit-shadowing conditions. This creates a strong relation between L-shadowing and expansiveness that we clarify with the results of this paper.
From \cite{CC2}*{Proposition 3} we know that topologically hyperbolic homeomorphisms admit the L-shadowing property.
In our first theorem we expand the class of dynamical systems admitting the L-shadowing property.

\begin{mainthm}\label{Examples}
Structurally stable diffeomorphisms and also some pseudo-Anosov diffeomorphisms of the two-dimensional sphere satisfy the $L$-shadowing property.
\end{mainthm}

This theorem is a consequence of Theorems \ref{non-wandering} and \ref{teoProjLsh}, which are interesting on their own.
On structurally stable diffeomorphisms, the local stable and unstable sets are transverse manifolds given by the Stable Manifold Theorem, while in the
pseudo-Anosov diffeomorphisms of the sphere, they are plaques of singular foliations with a finite number of singularities.
This example on the sphere was first considered by Walters in \cite{Walters} and further explored in \cites{PaVi,AAV,ArtigueDend}. Near each singularity, there are periodic hyperbolic horseshoes in arbitrarily small dynamical balls.
This is in contrast with the case of pseudo-Anosov homeomorphisms on surfaces of genus greater than 1, that do not satisfy the shadowing property (see \cite{LewowiczSurfaces}).

It is known that topologically hyperbolic homeomorphisms satisfy the Smale-Bowen spectral decomposition theorem, see for example \cite{AH}*{Theorem 3.4.4}.
From \cite{CC2}*{Proposition 2} we know that the L-shadowing property implies the classical shadowing property and that the non-wandering set decomposes as a finite union of chain recurrent classes.
Thus, the L-shadowing property allows us to recover some properties of topological hyperbolicity without assuming expansivity.
Notice that the pseudo-Anosov of the two-sphere of Theorem \ref{Examples} is transitive,
and consequently, a chain recurrent class (in this case the whole sphere) may not be expansive.
In our second main result we study the non-expansive chain recurrent classes of a homeomorphism satisfying the L-shadowing property,
proving that a similar phenomenon happening in the pseudo-Anosov diffeomorphism of the sphere is present in any non-expansive
chain recurrent class.

\begin{def*}
We say that a homeomorphism $f$ of a compact metric space $X$ admits \emph{arbitrarily small topological semi-horseshoes}
if for all $\epsilon>0$ there exist $N\geq 1$ and a compact, $f^N$-invariant set $K\subset X$
such that $\sup_{k\in\Z}\diam(f^k(K))\leq\epsilon$ and
$f^N\colon K\to K$ is semi-conjugate to a shift map.
\end{def*}

We use the terminology \emph{semi-horseshoe} since $K$ is semi-conjugate to a shift map and due to its relation to the phenomenon on the pseudo-Anosov diffeomorphism of the sphere, where indeed horseshoes appear.

\begin{mainthm}
\label{teoEspectralLsh}
If a homeomorphism, defined in a compact metric space, has the L-shadowing property, then it has the shadowing property, admits only a finite number of chain recurrent classes and each of its classes is either expansive or contains arbitrarily small topological semi-horseshoes.
\end{mainthm}
The idea behind the proof is the following: if the dynamical ball of some point $x$ is non-trivial, then the L-shadowing property assures the existence of a point in the dynamical ball of $x$ that is different from $x$ and is asymptotic to $x$ (Lemma \ref{lemLshByAsint}) allowing us to create arbitrarily small topological semi-horseshoes (Lemma \ref{lemaWillyBarney}) when $x$ is a non-wandering point.
This theorem characterizes the non-expansive phenomena that can exist in the non-wandering set of a homeomorphism satisfying the L-shadowing property.

Several generalizations of expansivity were considered before: the $n$-expansive systems in \cites{Art,APV,CC,CC2,LZ,Morales}, finite expansiveness in \cite{CC}, countable and measure expansivity in \cites{Art,CDX,ArCa}, cw-expansive homeomorphisms in \cites{Kato93,Kato93B} and entropy expansiveness in \cite{Bowen,PaVi} (among others).
The L-shadowing property was defined in \cite{CC2} as an attempt to link shadowing and Morales' $n$-expansivity \cite{Morales}.
From \cite{CC}*{Theorem A} we have that for every $n\in\N$ there is an $n$-expansive homeomorphism, defined
in a compact metric space, that is not $(n-1)$-expansive, has the shadowing property
and admits infinitely many chain recurrent classes.
By \cite{CC2}*{Proposition 2} cited above, these examples
do not satisfy the $L$-shadowing property.
Thus, shadowing and $n$-expansivity do not imply $L$-shadowing.

Notice that the topological semi-horseshoes defined above are, in particular, uncountable and have positive entropy. Thus, we obtain the following direct corollary of Theorem \ref{teoEspectralLsh}. We denote by $\Omega(f)$ the set of all non-wandering points of $f$.

\begin{mainclly}\label{CountableEntropy}
 If a homeomorphism $f$, defined in a compact metric space, has the L-shadowing property, then the following statements are equivalent:
 \begin{enumerate}
  \item $\Omega(f)$ is expansive,
  \item $\Omega(f)$ is countably-expansive,
  \item $\Omega(f)$ is entropy-expansive.
 \end{enumerate}
\end{mainclly}

Topological hyperbolicity implies the density of periodic points in the non-wandering set. It is not clear, tough, in which situations the L-shadowing property assures the same result. The expansive chain-recurrent classes contain a dense set of periodic points, but those containing arbitrarily small topological semi-horseshoes could not contain periodic points. 

\begin{mainthm}\label{aperiodic}
There exists a topologically mixing homeomorphism, defined in a compact metric space, satisfying the L-shadowing property and without periodic points.
\end{mainthm}

It is a consequence of Theorem B that this example admits arbitrarily small topological semi-horseshoes, tough in Remark \ref{rmkSemiFerrExplicitas} we exhibit them to clarify the definitions and results. In this work we obtain a characterization of the L-shadowing property in terms of local stable and unstable sets. For an expansive homeomorphism, the local stable (unstable) set of a point is contained in the stable (unstable) set of this point. However, for non-expansive homeomorphisms the local stable (unstable) set of a point can be much larger than the stable (unstable) set of this point. Indeed, several stable sets can intersect a same local stable set. Denote by $V^s_{\eps}(x)$ the intersection $W^s(x)\cap W^s_{\eps}(x)$ of the stable and the local stable sets of $x$ and by $V^u_{\eps}(x)$ the intersection $W^u(x)\cap W^u_{\eps}(x)$ of the unstable and the local unstable sets of $x$.

\begin{mainthm}
\label{teoCharLsh}
A homeomorphism, defined in a compact metric space, has the L-shadowing property if, and only if, it has the shadowing property and satisfies: for each $\epsilon>0$ there is $\delta>0$ such that $\dist(x,y)<\delta$
 implies $V^s_\epsilon(x)\cap V^u_\epsilon(y)\neq\emptyset$.
\end{mainthm}

This last property is similar to the local product structure satisfied by topologically hyperbolic homeomorphisms, but it holds in more general scenarios where the local stable sets have much more complicated behavior.
Before proving the L-shadowing property, we show the limit shadowing property (Theorem \ref{teoLimSh}) that is an interesting result on its own. We also consider the relation between the L-shadowing property and the two-sided limit shadowing property discussed in \cites{Cthesis,C,C2,CC,CC2,CK}. It is proved that topologically mixing homeomorphisms admitting the L-shadowing property also satisfy the two-sided limit shadowing property (see Proposition \ref{mixing}) but the converse is not clear (see Question \ref{tslsandLs}). 
Notice that the example of Theorem \ref{aperiodic} is topologically mixing and satisfies both the L-shadowing and the two-sided limit shadowing properties. The chain recurrent classes of a homeomorphism satisfying the L-shadowing property are not necessarily topologically mixing, even in the expansive case, but they are transitive since they satisfy the shadowing property. In terms of the shadowing theory, this says that they do not necessarily satisfy the two-sided limit shadowing property. However, we can obtain a decomposition of each class in periodic elementary sets satisfying the two-sided limit shadowing property, as in the Bowen decomposition in the hyperbolic case.

\begin{mainthm}\label{Bowen}
If a homeomorphism $f$, defined in a compact metric space $X$, satisfies the L-shadowing property and $C$ is a chain recurrent class of $f$, then there exist $n\geq 1$, $C_1,\dots,C_n\subset C$, compact, disjoint, $f^n$-invariant sets such that $C=\bigcup_{i=1}^nC_i$, $f(C_i)=C_{(i+1)\mod n}$ and $f^n$ restricted to each $C_i$ is topologically mixing and satisfies the two-sided limit shadowing property.
\end{mainthm}

In the last result of this paper, we generalize the main theorems of \cite{CC2} on positively expansive homeomorphisms. A classical result in topological dynamics is that if a positively expansive homeomorphism is defined in a compact metric space, then this space must be finite (see \cite{KR} for example). This theorem does not hold when we consider generalizations of positive expansivity: Morales in \cite{Morales} proved that the minimal subset of the classical Denjoy homeomorphism of $\mathbb{S}^1$ is positively 2-expansive and is defined in a Cantor subset, so it is not positively expansive. Examples of positively $n$-expansive homeomorphisms for each $n\in\N$ were introduced by Li and Zhang in \cite{LZ}, modifying a little bit the minimal subset of the Denjoy homeomorphism. It is proved \cite{CC2} that transitive positively $n$-expansive homeomorphisms satisfying the shadowing property can only be defined in finite metric spaces. The same is proved assuming the L-shadowing property (see \cite{CC2}*{Theorem A}). Finally, we generalize these results for positively finite-expansive homeomorphisms satisfying the shadowing property. A homeomorphism is said to be \emph{positively finite-expansive} if there exists $c>0$ such that the local stable set $W^s_c(x)$ of every $x\in X$ is finite.


\begin{mainthm}\label{h}
If a positively finite-expansive homeomorphism is defined in a compact metric space $X$ and has the shadowing property, then
$X$ is finite.
\end{mainthm}

The paper is organized as follows: in Section \ref{secCharLsh} we characterize the L-shadowing property proving Theorem \ref{teoCharLsh}, discuss the relation between the L-shadowing and the two-sided limit shadowing properties and prove Theorem \ref{Bowen}; in Section 3 we show how to construct the topological semi-horseshoes using the shadowing property and prove Theorem \ref{teoEspectralLsh}; in Section 4 we prove the examples of Theorems \ref{Examples} and \ref{aperiodic} and in Section 5 we prove Theorem \ref{h} on positive expansivity.

\vspace{+0.4cm}

\section{Characterization of L-shadowing}
\label{secCharLsh}
In this section we will prove Theorem \ref{teoCharLsh}, a characterization for the L-shadowing property in terms of the sets $V_{\eps}^s$ and $V_{\eps}^u$ and the shadowing property. We recall the definition of the classical shadowing property. Through this whole section, $(X,d)$ denotes a compact metric space and $f\colon X\to X$ a homeomorphism.

\begin{definition}
A sequence $(x_k)_{k\in\Z}\subset X$ is called a \emph{$\delta$-pseudo-orbit} if
it satisfies $$d(f(x_k),x_{k+1})<\delta \,\,\,\,\,\, \text{for every} \,\,\,\,\,\, k\in\Z.$$
The sequence $(x_k)_{k\in\Z}\subset X$ is $\eps$-\emph{shadowed} if there exists $y\in X$ satisfying $$d(f^k(y),x_k)<\eps \,\,\,\,\,\, \text{for every} \,\,\,\,\,\, k\in\Z.$$ We say that $f$ has the \emph{shadowing property} if for every $\eps>0$ there exists $\delta>0$ such that every $\delta$-pseudo-orbit is
$\eps$-shadowed.
\end{definition}

\begin{rmk}
\label{rmkLshCharDirectTrivial}By \cite{CC2}*{Proposition 2} we know that
the L-shadowing property implies the shadowing property.
Then, the direct part of Theorem \ref{teoCharLsh} follows observing that the past orbit of $y$ and the future orbit of $x$ form a $(\delta,L)$-pseudo orbit when $d(x,y)<\delta$ and the $(\eps,L)$-shadowing relation is equivalent, in this case, to $V^s_{\eps}(x)\cap V^u_{\eps}(y)\neq\emptyset$.
\end{rmk}

To prove the converse of Theorem \ref{teoCharLsh}, we first prove the limit shadowing property in Theorem \ref{teoLimSh}. This property was introduced by Eirola, Nevanlinna and Pilyugin in \cite{ENP}, see also \cite{P1}.

\begin{def*}
A sequence $(x_k)_{k\in\N}\subset X$ is called a \emph{limit pseudo-orbit} if it
satisfies $$d(f(x_k),x_{k+1})\rightarrow 0 \,\,\,\,\,\, \text{when} \,\,\,\,\,\, k\rightarrow\infty.$$
The sequence $\{x_k\}_{k\in\N}$ is \emph{limit-shadowed} if there exists $y\in X$ such that $$d(f^k(y),x_k)\rightarrow 0, \,\,\,\,\,\, \text{when} \,\,\,\,\,\,k\rightarrow
\infty.$$ We say that $f$ has the \emph{limit shadowing property} if every limit pseudo-orbit is limit-shadowed.
\end{def*}

The shadowing property assures the existence of points shadowing a limit pseudo orbit $(x_k)_{k\in\N}$ with any given desired accuracy, if we consider sufficiently large iterations. Indeed, if $(\delta_n)_{n\in\N}$ is any sequence of numbers converging to zero, then the shadowing property assures the existence of $(p_n)_{n\in\N}\subset X$ and an increasing sequence $(k_n)_{n\in\N}$ of natural number such that $(x_k)_{k\geq k_n}$ is $\delta_n$-shadowed by $f^{k_n}(p_n)$. It is natural to ask if any limit point $p^*$ of the sequence $(p_n)_{n\in\N}$ limit shadows $(x_k)_{k\in\N}$. 
For a fixed $k\in\N$ we analyze the number $d(f^k(p^*),x_k)$ that is the limit of the real sequence $d(f^k(p_n),x_k)$ when $n\to\infty$. Note that when $n$ is sufficiently large, we have $k<k_n$ and, hence, $d(f^k(p_n),x_k)$ is not necessarily bounded by $\delta_n$.

Using that $V^u_{\eps}(x)$ and $V^s_{\eps}(y)$ intersect when $x$ and $y$ are sufficiently close, we can solve the previous situation intersecting successively carefully chosen iterates of the points $p_n$. This choice is done in the next lemma, that is the induction step needed in the proof of Theorem \ref{teoLimSh}.

\begin{lemma}
\label{lemLimSh}
Suppose that $f$ has the shadowing property, $c>0$ is given and $\delta>0$ is such that
$\dist(a,b)<2\delta$ implies $V^s_{c}(a)\cap V^u_{c}(b)\neq\emptyset$,
$(x_k)_{k\in\N}$ is a limit pseudo orbit of $f$
and that $p\in X$ is such that $\dist(f^k(p),x_k)<\delta$ for all $k\geq 0$. Then, for all $\epsilon\in (0,\delta)$ there are $q\in X$ and $M<N\in\N$ arbitrarily large such that
\[
 \dist(f^k(q),x_k)\leq\left\{
\begin{array}{ll}
 c+\delta &\text{ if }\,\,\,0\leq k\leq M,\\
 c+\epsilon &\text{ if }\,\,\,M\leq k\leq N,\\
 \epsilon & \text{ if }\,\,\, k\geq N.
\end{array}
 \right.
\]
\end{lemma}

\begin{proof}
Since $f$ has the shadowing property,
for each $\epsilon\in (0,\delta)$
there are $y\in X$ and $M\in\N$ such that
\[ \dist(f^{k}(y),x_{k})<\frac{\epsilon}{2}\,\,\,\,\,\,\text{ for all }\,\,\,\,\,\,k\geq M.
\]
For $k\geq M$ we have
\begin{eqnarray*}
\dist(f^k(y),f^{k}(p))&\leq& \dist(f^k(y),x_{k}) + \dist(x_{k},f^{k}(p))\\
&<&\epsilon/2+\delta<2\delta.
\end{eqnarray*}
Then there exists $q\in X$ such that $$f^M(q)\in V^s_{c}(f^{M}(y))\cap V^u_{c}(f^{M}(p)).$$
If $0\leq k\leq M$ then $d(f^k(q),f^k(p))<c$ since $f^M(q)\in V^u_{c}(f^{M}(p))$ and, hence,
\begin{eqnarray*}
 \dist(f^k(q),x_k)&\leq& \dist(f^k(q),f^k(p)) + \dist(f^k(p),x_k)\\
&<&c+\delta.
\end{eqnarray*}
For $k\geq M$, we have $\dist(f^{k}(q),f^{k}(y))<c$ since  $f^M(q)\in V^s_{c}(f^{M}(p))$ and, hence,
\begin{eqnarray*}
\dist(f^{k}(q),x_{k})&\leq& \dist(f^{k}(q),f^{k}(y)) + \dist(f^{k}(y),x_{k})\\
&<&c+\epsilon/2<c+\epsilon.
\end{eqnarray*}
Now choose $N\geq M$ such that
$\dist(f^k(q),f^k(y))<\epsilon/2$ for all $k\geq N$.
Then
\begin{eqnarray*}
\dist(f^k(q),x_k)&\leq&\dist(f^k(q),f^k(y))+\dist(f^k(y),x_k)\\
 &<&\epsilon/2+\epsilon/2=
  \epsilon
 \end{eqnarray*}
 for all $k\geq N$.
\end{proof}

\begin{theorem}
 \label{teoLimSh}
  If $f$ has the shadowing property and
  for all $\epsilon>0$ there is $\delta>0$ such that $\dist(x,y)<\delta$
 implies $V^s_\epsilon(x)\cap V^u_\epsilon(y)\neq\emptyset$, then
 $f$ has the limit shadowing property.
\end{theorem}

\begin{proof}
Let $(c_n)_{n\in\N}$ be a sequence of positive numbers such that $\sum_{n=0}^\infty c_n<\infty$ and choose a decreasing sequence $\delta_n\to 0$ such that $$\dist(a,b)<2\delta_n \,\,\,\,\,\,\text{implies} \,\,\,\,\,\, V^s_{c_n}(a)\cap V^u_{c_n}(b)\neq\emptyset.$$
Suppose that $(x_k)_{k\in\N}$ is a limit pseudo orbit of $f$.
Assume that $\dist(f^k(p_0),x_k)<\delta_0$ for all $k\geq 0$.
We will prove that there exist $(p_n)_{n\in\N}\subset X$ and integers $M_0=N_0=0<M_1<N_1<M_2<N_2<\dots$ such that
for all $n>l\geq 0$ we have
\[
\dist(f^k(p_n),x_k)\leq
\left\{
\begin{array}{ll}
\delta_l+\sum_{i=l}^{n-1} c_i  & \text{ if }\,\,\, N_l\leq k\leq M_{l+1},\\
\delta_{l+1}+\sum_{i=l}^{n-1} c_i  & \text{ if }\,\,\, M_{l+1}\leq k\leq N_{l+1},\\
\delta_n  &  \text{ if }\,\,\, k\geq N_n.
\end{array}
\right.
\]
By Lemma \ref{lemLimSh} there are $p_1\in X$ and $N_1>M_1>0$ such that
\[
 \dist(f^k(p_1),x_k)<\left\{
\begin{array}{ll}
 c_0+\delta_0 &\text{ if }\,\,\,0\leq k\leq M_1,\\
 c_0+\delta_1 &\text{ if }\,\,\,M_1\leq k\leq N_1,\\
 \delta_1 & \text{ if } \,\,\,k\geq N_1.
\end{array}
 \right.
\]
Applying again Lemma \ref{lemLimSh} we obtain $p_2\in X$ and $N_2>M_2>N_1$ such that
\[
 \dist(f^k(p_2),x_k)<\left\{
\begin{array}{ll}
 c_1+c_0+\delta_0 &\text{ if }\,\,\,0\leq k\leq M_1,\\
 c_1+c_0+\delta_1 &\text{ if }\,\,\,M_1\leq k\leq N_1,\\
 c_1+\delta_1 &\text{ if }\,\,\,N_1\leq k\leq M_2,\\
 c_1+\delta_2 &\text{ if }\,\,\,M_2\leq k\leq N_2,\\
 \delta_2 & \text{ if } \,\,\,k\geq N_2.
\end{array}
 \right.
\]
In this way, by induction we define $p_n,M_n,N_n$ as claimed.
Let $p_*$ be a limit point of $(p_n)_{n\in\N}$.
For each $k\geq 0$ there is $l_k$ such that $N_{l_k}\leq k\leq N_{l_k+1}$
and then
\[
\dist(f^k(p_*),x_k)\leq \delta_{l_k}+\sum_{i=l_k}^{\infty} c_i.
\]
If $k\to\infty$ then $l_k\to\infty$, $\delta_{l_k}\to 0$  and $\sum_{i=l_k}^{\infty} c_i\to 0$.
Thus, $\dist(f^k(p_*),x_k)\to 0$.
\end{proof}


\begin{proof}[Proof of Theorem \ref{teoCharLsh}]
As explained in Remark \ref{rmkLshCharDirectTrivial}, we only have to prove the converse of the theorem.
For each $\beta>0$ let $\eps\in(0,\beta/3)$ be such that
$$d(x,y)<\eps\,\,\,\text{ implies }\,\,\,V^u_{\beta/3}(x)\cap V^s_{\beta/3}(y)\neq\emptyset.$$
Choose $\delta>0$, given by the shadowing property, such that every $\delta$-pseudo orbit is $\epsilon/3$-shadowed.
We will prove that each $(\delta,L)$-pseudo orbit $(x_k)_{k\in\Z}\subset X$ of $f$ is $(\beta,L)$-shadowed. By Theorem \ref{teoLimSh}, $f$ satisfies the limit shadowing property, so there exist $p_s,p_u\in X$ that limit shadow $(x_k)_{k\in\N}$ in the future and in the past, respectively.
Also, the shadowing property assures the existence of $y\in X$ that $\epsilon/3$-shadows $(x_k)_{k\in\Z}$.
If $n\in\N$ is big enough, then $d(f^n(y),f^n(p_s))<\eps$ and, hence, there exists $$z_1\in V^u_{\beta/3}(f^n(y))\cap V^s_{\beta/3}(f^n(p_s)).$$
Choose $m\in\N$ such that
$$d(f^{-m}(y),f^{-m}(p_u))<\frac{\epsilon}{2} \,\,\,\,\,\,\text{and} \,\,\,\,\,\,d(f^{-n-m}(z_1),f^{-m}(y))<\frac{\eps}{2}.$$
This implies that
$d(f^{-m}(p_u),f^{-n-m}(z_1))<\eps$
and then there is $$z_2\in V^u_{\beta/3}(f^{-m}(p_u))\cap V^s_{\beta/3}(f^{-n-m}(z_1)).$$
Thus, $z=f^m(z_2)$ satisfies:
\begin{enumerate}
	\item $z\in W^u(p_u)$ since $f^{-m}(z)=z_2\in V^u_{\beta/3}(f^{-m}(p_u))$,
	\item $z\in W^s(p_s)$ since $f^{-m}(z)=z_2\in V^s_{\beta/3}(f^{-n-m}(z_1))$ and
	$z_1\in V^s_{\beta/3}(f^n(p_s))$,
	\item $z\in \Gamma_{2\beta/3}(y)$ since $z_1\in V^u_{\beta/3}(f^n(y))$ and $z_2\in  V^s_{\beta/3}(f^{-n-m}(z_1))$.
\end{enumerate}
Items (1), (2) and (3) imply that $(x_k)_{k\in\Z}$ is $(\beta,L)$-shadowed by $z$ and the $L$-shadowing property is proved.
\end{proof}

The limit shadowing property considers sequences indexed by the natural numbers, but it has an analogue considering bilateral sequences, called \emph{two-sided limit shadowing property}.
Information about this property can be found in \cite{Cthesis,C,C2,CC,CC2,O,P1}.

\begin{definition}
A sequence $(x_k)_{k\in\Z}$ is a \emph{two-sided limit pseudo-orbit} if it satisfies $$d(f(x_k),x_{k+1})\rightarrow 0, \,\,\,\,\,\, |k|\rightarrow\infty.$$
The sequence $(x_k)_{k\in\Z}$ is \emph{two-sided limit shadowed} if there exists $y\in X$ satisfying $$d(f^k(y),x_k)\rightarrow 0, \,\,\,\,\,\,
|k|\rightarrow \infty.$$ We say that $f$ has the \emph{two-sided limit shadowing property} if every two-sided limit pseudo-orbit is two-sided limit shadowed.
\end{definition}

This property has one similarity with the L-shadowing property, that is, while the limit shadowing property for $f$ and its inverse assure the existence of one point limit-shadowing $(x_k)_{k\in\Z}$ in the future and another one in the past, in the two-sided limit shadowing property we obtain a single point with both behaviors. Then a similar result with Theorem \ref{teoCharLsh} is obtained.

\begin{proposition}
A homeomorphism $f$ of a compact metric space $X$ has the two-sided limit shadowing property if, and only if, $f$ and $f^{-1}$ have the limit shadowing property and $W^u(x)\cap W^s(y)\neq\emptyset$ for every $x,y\in X$.
\end{proposition}

The proof is clear, though one can see \cite{Cthesis}*{Lemma 1}. The two-sided limit shadowing property implies both shadowing and topological mixing as is proved in \cite{CK} (recall that $f$ is topologically mixing if for any pair of non-empty open subsets $U,V\subset X$ there exists $n>0$ such that $f^k(U)\cap V\neq\emptyset$ for every $k\geq n$). The converse is not true, tough in \cite{C}*{Lemma 2.2} one can find a proof assuming expansiveness. In the following proposition, instead of expansiveness, we use the L-shadowing property.

\begin{proposition}\label{mixing}
If a topologically mixing homeomorphism is defined in a compact metric space and has the L-shadowing property, then it has the two-sided limit shadowing property.
\end{proposition}

\begin{proof}
To prove the two-sided limit shadowing property, it is enough to prove that $W^u(x)\cap W^s(y)\neq\emptyset$ for every $x,y\in X$ as in the previous proposition. Indeed, it is simple to note that the L-shadowing property implies the limit shadowing property for both $f$ and $f^{-1}$. Also note that $f$ has the specification property since it has the shadowing property and is topologically mixing (see \cite{DGS}). Let $\eps=\diam(X)$ and consider $\delta>0$, given by the L-shadowing property, such that every $(\delta,L)$-pseudo orbit is $(\eps,L)$-shadowed. Let $m\in\N$, given by the specification property, be such that every $m$-spaced specification is $\delta$-shadowed. The specification property assures the existence of $w\in B(f^{-m}(x),\delta)$ such that $f^{2m}(w)\in B(f^m(y),\delta)$. Thus, the sequence formed by the past orbit of $f^{-m}(x)$, the segment of orbit from $w$ to $f^{2m}(w)$ and the future orbit of $f^m(y)$ is a $(\delta,L)$-pseudo orbit of $f$. Then the L-shadowing property assures that it is $(\eps,L)$-shadowed by $z\in X$ and, hence, $z\in W^u(x)\cap W^s(y)$.
\end{proof}

It is not clear whether the two-sided limit shadowing property implies the L-shadowing property. It could happen that the point two-sided limit shadowing a given $(\delta,L)$-pseudo orbit is very distant from the pseudo orbit during some iterates, while the point $\eps$-shadowing it can not limit shadow it. So the following question is still unanswered.

\begin{question}\label{tslsandLs}
Does the two-sided limit shadowing property imply the L-shadowing property?
\end{question}


The \emph{chain-recurrent class} of $x\in X$ is the set of all $y\in X$ such that for every $\eps>0$ there exist a periodic $\eps$-pseudo orbit containing both $x$ and $y$. We say that $f$ is \emph{transitive} if for any pair of non-empty open subsets $U,V\subset X$, there exists $n\in\N$ such that $f^n(U)\cap V\neq\emptyset$. It is easy to see that transitive homeomorphisms admit only one chain recurrent class, that is the whole space. Now we prove Theorem \ref{Bowen} where a decomposition of each class in elementary sets satisfying the two-sided limit shadowing property is also obtained.

\begin{proof}[Proof of Theorem \ref{Bowen}]
We know that $f$ admits only a finite number of chain recurrent classes, so the restriction of $f$ to each of its classes is a transitive homeomorphism satisfying the L-shadowing property. 
If $f_{|C}$ is topologically mixing, then Proposition \ref{mixing} assures that it has the two-sided limit shadowing property and the whole class $C$ is an elementary set as in the theorem. If $f_{|C}$ is not topologically mixing, then \cite{KKO}*{Theorem 3.8} implies that some iterate $f_{|C}^m$ is not transitive. Then \cite{B}*{Corollary 2.1} assures the existence of $n$ dividing $m$ and sets $C_1,\dots, C_n\subset C$, compact and $f^n$-invariant such that $C=\bigcup_{i=1}^nC_i$ and $f(C_i)=C_{(i+1)\mod n}$. These sets are disjoint by \cite{O}*{Lemma 4} and $f^n_{|C_i}$ is totally transitive by \cite{B}*{Theorem 3.1}. Then $f^n_{|C_i}$ is topologically mixing (again by \cite{KKO}*{Theorem 3.8}) and Proposition \ref{mixing} assures that it satisfies the two-sided limit shadowing property.
\end{proof}

\section{Semi-horseshoes}
In this section, we prove Theorem \ref{teoEspectralLsh}, a Spectral Decomposition Theorem characterizing the non-expansive chain recurrent classes of homeomorphisms with the L-shadowing property, as those containing arbitrarily small topological semi-horseshoes.
As we said, the shadowing property and the finiteness of the chain recurrent classes were proved in \cite{CC2}.
Then, we turn our attention to the last part of the theorem and start with a technical lemma, where, under the assumption of the shadowing property, a sufficient condition to the existence of topological semi-horseshoes is obtained. In \cite{SuOp} a similar result was obtained in a slightly different context.

We recall that a point $x\in X$ is called a \emph{non-wandering point} if for each open subset $U$ of $X$ containing $x$, there is $k>0$ such that $f^k(U)\cap U\neq\emptyset$. The set of all non-wandering points of $f$ is called the \emph{non-wandering set} and is denoted by $\Omega(f)$.

\begin{lemma}
\label{lemaWillyBarney}
Let $f\colon X\to X$ be a homeomorphism satisfying the shadowing property.
For all $\eps>0$ there is $\delta>0$ such that if $x\in\Omega(f)$, $y\in X$, $n>0$ satisfy:
\begin{equation}
 \label{ecuLiYorkeTrucho}
 \left\{
 \begin{array}{l}
 \epsilon<\max\{\dist(f^k(x),f^k(y)):0\leq k< n\}=:\gamma,\\
 \max\{\dist(x,y),\dist(f^n(x),f^n(y))\}<\delta,
 \end{array}
 \right.
\end{equation}
then
there is $N\geq 1$ and a compact set $K\subset X$ such that $\sup_{k\in\Z}\diam(f^k(K))\leq 2\gamma$,
$f^N(K)=K$ and
$f^N\colon K\to K$ is semi-conjugate to a shift of two symbols.
In particular, $K$ is uncountable and $h(K)\geq\frac{\log(2)}N$.
\end{lemma}

\begin{proof}
Given $\epsilon>0$, by the shadowing property there is $\delta>0$ such that
every $2\delta$-pseudo orbit can be $\epsilon/4$-shadowed.
Suppose that $x\in\Omega(f)$, $y\in X$ and $n>0$ satisfy \eqref{ecuLiYorkeTrucho}.
Since $x\in\Omega(f)$, there are $j\geq 0$ and $z\in B(f^n(x),\delta)$ such that
$f^j(z)\in B(x,\delta)$.
Let $w\in\{x,y\}^\Z$ be a sequence with $w_k\in\{x,y\}$ for all $k\in\Z$.
We define a sequence $(\tilde w_k)_{k\in\Z}$ such that if $k=q(n+j)+r$ with $0\leq r<n+j$ then
\[
 \tilde w_k=\left\{
 \begin{array}{ll}
  f^r(w_q) & \text{if }0\leq r<n,\\
  f^{r-n}(z) & \text{if }n\leq r<n+j.
 \end{array}
 \right.
\]
In Figure \ref{figShadShift} we illustrate such sequences.
\begin{figure}[h]
 \includegraphics{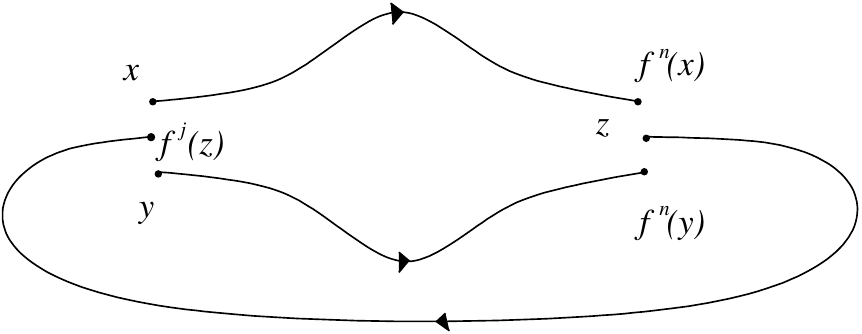}
 \caption{A sequence $\tilde w$ corresponds to an infinite cycle of this diagram.
 Each $w_k=x$ or $y$, indicates that at the $k$-turn the sequence follows the orbit segment of $x$ or $y$.}
 \label{figShadShift}
\end{figure}

Notice that each $\tilde w$ is a $2\delta$-pseudo orbit and if $w\neq v$ then
\[
 \epsilon<\sup_{k\in\Z}\dist(\tilde w_k,\tilde v_k)\leq\gamma.
\]
For $w\in\{x,y\}^\Z$ consider the set
\[
K_w=\left\{z\in X: \sup_{k\in\Z}\dist(f^k(z),\tilde w_k)\leq\epsilon/4\right\}.
\]
The shadowing property assures that each $K_w$ is not empty.
If $w\neq v$, $p\in K_w$ and $q\in K_v$ then
\[
\begin{array}{ll}
 \epsilon/2  & =-\epsilon/4+\epsilon-\epsilon/4\\
 & \leq\sup_{k\in\Z}-\dist(f^k(p),\tilde w_k)+\dist(\tilde w_k,\tilde v_k)-\dist(\tilde v_k,f^k(q))\\
 & \leq\sup_{k\in\Z}\dist(f^k(p),f^k(q))\\
 & \leq\sup_{k\in\Z}\dist(f^k(p),\tilde w_k)+\dist(\tilde w_k,\tilde v_k)+\dist(\tilde v_k,f^k(q))\\
 &\leq\epsilon/4+\gamma+\epsilon/4=\gamma+\epsilon/2<2\gamma
\end{array}
\]
and thus
\[
 \epsilon/2 \leq\sup_{k\in\Z}\dist(f^k(p),f^k(q))<2\gamma.
\]
In particular, $K_w\cap K_v=\emptyset$ if $w\neq v$. Define
\[
\displaystyle K=\bigcup_{w\in\{x,y\}^\Z} K_w.
\]
Note that $\diam(f^k(K))\leq 2\gamma$ for all $k\in\Z$ and
consider the map $h\colon K\to \{x,y\}^\Z$ given by $h(p)=w$ if $p\in K_w$.
It follows that $h\circ f^N=\sigma\circ h$, where $N=n+j$ and
$\sigma\colon\{x,y\}^\Z\to \{x,y\}^\Z$ is the shift homeomorphism.
We leave the remaining details to the reader (which are analogous to \cite{SuOp}*{Theorem 5.1}).
%
\end{proof}

\begin{rmk}
For a $C^\infty$ diffeomorphism $f$ of a smooth manifold, the entropy of the set $K$ given in Lemma \ref{lemaWillyBarney}
approaches zero as $\epsilon\to 0$.
This is due to a result of Buzzi \cite{Buzzi}*{Theorem 2.2}.
\end{rmk}

The following lemma, is crucial in our proof of Theorem \ref{teoEspectralLsh} and is a step where the L-shadowing property is indeed important. If some point $x\in X$ has a non-trivial dynamical ball, then the L-shadowing property assures that it also has a non-trivial asymptotic dynamical ball. We recall that the dynamical ball of $x\in X$ of radius $\delta>0$ is the set $$\Gamma_{\delta}(x):=\{y\in X \,\,\, ; \,\,\, d(f^n(x),f^n(y))\leq \delta \,\,\, \text{for every} \,\,\, n\in\Z\}$$ and define the \emph{asymptotic dynamical ball of $x$} of radius $\eps$ as the set $V^s_\epsilon(x)\cap V^u_\epsilon(x)$.

\begin{lema}
\label{lemLshByAsint} If $f$ has the L-shadowing property, then for all $\epsilon>0$ there exists $\delta>0$
 such that for any $x\in X$
 \[
  V^s_\epsilon(x)\cap V^u_\epsilon(x)=\{x\} \,\,\,\,\,\, \text{implies} \,\,\,\,\,\, \Gamma_\delta(x)=\{x\}.
 \]
\end{lema}

\begin{proof}
For each $\epsilon>0$, there exists $\rho\in (0,\epsilon/2)$ given by Theorem \ref{teoCharLsh} such that
$$\dist(a,b)<\rho \,\,\,\,\,\, \text{implies} \,\,\,\,\,\, V^s_{\epsilon/4}(a)\cap V^u_{\epsilon/4}(b)\neq\emptyset.$$
Also, there is $\delta\in(0,\rho)$ such that $$\dist(a,b)<\delta \,\,\,\,\,\, \text{implies} \,\,\,\,\,\, V^s_\rho(b)\cap V^u_\rho(a)\neq\emptyset.$$
Suppose that there is $y\in \Gamma_\delta(x)\setminus\{x\}$
and let $u\in V_\rho^s(y)\cap V_\rho^u(x)$.
If $u=x$ then $y\in V^s_\rho(x)\cap\Gamma_\delta(x)$.
If $u\neq x$ then $u\in V^u_\rho(x)\cap\Gamma_\rho(x)$.
In any case, we obtain a point $z$ ($z=u$ or $y$) different from $x$ belonging to $V^\sigma_\rho(x)\cap\Gamma_\rho(x)$, for $\sigma=s$ or $u$.
Considering $f^{-1}$ if needed, we assume that
\begin{equation}
 \label{ecuLimShPtoz}
 z\in V^u_\rho(x)\cap\Gamma_\rho(x)\setminus\{x\}.
\end{equation}

If $\liminf_{n\to +\infty} \dist(f^n(x),f^n(z))=0$ then the L-shadowing finishes the proof as follows.
Let $r=d(x,z)$ and choose $s>0$ given by the L-shadowing property for $r/2$. Let $k\in\N$ be such that
$$d(f^{-k}(z),f^{-k}(x))<s \,\,\,\,\,\, \text{and} \,\,\,\,\,\, d(f^k(z),f^k(x))<s$$ and note that the past orbit of $f^{-k}(x)$, the
segment of orbit from $f^{-k}(z)$ to $f^k(z)$ and the future orbit of $f^k(x)$ is a $(s,L)$-pseudo orbit of $f$ and, hence, is $(r,L)$-shadowed by $p\in X$.
Then $p\in V^s_\epsilon(x)\cap V^u_\epsilon(x)$ and $p\neq x$ because $d(p,z)<r/2$.

Then, we can choose $\alpha>0$ such that $\dist(f^n(x),f^n(z))>\alpha$ for all $n\geq 0$.
Let $\mu\in (0,\epsilon/2)\cap (0,\alpha/3)$
and consider $\gamma\in (0,\alpha/3)\cap (0,\epsilon/4)$ such that
every $(2\gamma,L)$-pseudo orbit is $(\mu,L)$-shadowed.
Note that $\alpha-2\gamma-\mu>0$.
As $X$ is compact, there are $n_k\to +\infty$ and $x^*,z^*\in X$ such that
$f^{n_k}(x)\to x^*$ and $f^{n_k}(z)\to z^*$.
The continuity of $f$ assures that
\[\alpha\leq\dist(f^i(x^*),f^i(z^*))\leq\rho \,\,\,\,\,\, \text{for all} \,\,\,\,\,\, i\in\Z. \]
This implies that there is
\begin{equation}
 \label{ecuLShw}
 w\in V^s_{\epsilon/4}(x^*)\cap V^u_{\epsilon/4}(z^*).
\end{equation}
In Figure \ref{figShadLim} the situation is illustrated.
\begin{figure}[h]
 \includegraphics{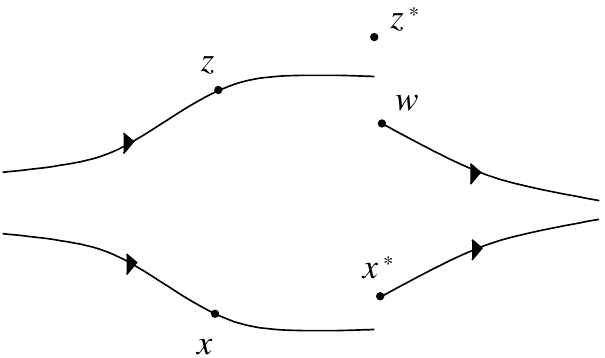}
 \caption{Illustration for the proof of Lemma \ref{lemLshByAsint}.}
 \label{figShadLim}
\end{figure}

Let $\ell\geq 1$ be such that $$\dist(f^{\ell}(w),f^{\ell}(x^*))<\gamma \,\,\,\,\,\, \text{and} \,\,\,\,\,\, \dist(f^{-\ell}(w),f^{-\ell}(z^*))<\gamma$$
and choose $m\geq 1$ such that
\begin{equation}
 \label{ecuLimShm}
\left.
\begin{array}{l}
\dist(f^{m+j}(x),f^j(x^*))<\gamma\\
\dist(f^{m-j}(z),f^{-j}(z^*))<\gamma
\end{array}
\right\}
\text{ whenever }|j|\leq \ell.
\end{equation}
Consider the $2\gamma$-pseudo orbit
\[
 p_i=
 \left\{
 \begin{array}{ll}
  f^i(z) & \text{ if }\,\, i<m-\ell,\\
  f^{i-m}(w) & \text{ if } \,\,m-\ell\leq i<m+\ell, \\
  f^i(x) & \text{ if }\,\,m+\ell\leq i.\\
 \end{array}
 \right.
\]
Then, there is $q\in W^s(x)\cap W^u(z)$ such that
$\dist(f^i(q),p_i)<\mu$ for all $i\in\Z$.
As $z\in W^u(x)$ it follows that
$q\in W^s(x)\cap W^u(x)$. Note that $q\in\Gamma_\epsilon(x)$ since for each $i\in\Z$ we have
\begin{eqnarray*}
 \dist(f^i(q),f^i(x))&\leq&
  \dist(f^i(q),p_i)+\dist(p_i,f^i(x))\\
&<& \mu+\dist(p_i,f^i(x))\\
&<&\frac{\epsilon}{2}+\dist(p_i,f^i(x)).
\end{eqnarray*}
 Also, note that $\dist(p_i,f^i(x))<\epsilon/2$ for all $i\in\Z$ since for $i\geq m+\ell$ we have $p_i=f^i(x)$,
if $i<m-l$, then $p_i=f^i(z)$ and by \eqref{ecuLimShPtoz} we conclude that
$$\dist(f^i(x),f^i(z))\leq\rho<\epsilon/2$$
and if $m-l\leq i<m+l$, then $p_i=f^{i-m}(w)$ and by \eqref{ecuLShw} and \eqref{ecuLimShm} we obtain
\begin{eqnarray*}
 \dist(f^i(x),f^{i-m}(w))&\leq& \dist(f^i(x),f^{i-m}(x^*))+\dist(f^{i-m}(x^*),f^{i-m}(w))\\
 &\leq& \gamma +\frac{\epsilon}{4}<\frac{\epsilon}{2}.
\end{eqnarray*}
To prove that $q\neq x$ note that
\begin{eqnarray*}
 \dist(f^{m-\ell}(x),f^{m-\ell}(q))&\geq& \dist(f^{m-\ell}(x),p_{m-\ell})-\dist(p_{m-\ell},f^{m-\ell}(q))\\
 &>& \dist(f^{m-\ell}(x),f^{-\ell}(w))-\mu\\
  &\geq&\dist(f^{-\ell}(x^*),f^{-\ell}(z^*))-
 \dist(f^{-\ell}(x^*),f^{m-\ell}(x))\\
 & &-\dist(f^{-\ell}(w),f^{-\ell}(z^*))-\mu\\
& \geq& \alpha-2\gamma-\mu>0.
\end{eqnarray*}
This finishes the proof.
\end{proof}

\begin{proof}[Proof of Theorem \ref{teoEspectralLsh}]
Let $f$ be a homeomorphism satisfying the $L$-shadowing property. Since it admits a finite number of chain recurrent classes, there exists $r>0$ such that the $2r$-neighborhoods of all its chain recurrent classes are disjoint. This also implies that each class is the maximal invariant set of its $2r$-neighborhood. Let $\eps$ be an arbitrary number in $(0,r)$ and consider $\delta\in(0,\eps)$, given by Lemma \ref{lemLshByAsint}, such that $$\Gamma_\delta(x)\neq\{x\}\,\,\,\,\,\, \text{implies} \,\,\,\,\,\, V^s_\epsilon(x)\cap V^u_\epsilon(x)\neq\{x\}.$$
Let $\Lambda\subset X$ be a chain recurrent class of $f$ and assume it is not expansive. Then there exists $x\in\Lambda$ and $y\in\Gamma_{\delta}(x)\setminus\{x\}$. Lemma \ref{lemLshByAsint} assures the existence of $$z\in V^s_\epsilon(x)\cap V^u_\epsilon(x)\setminus\{x\}.$$ Let $\eps_1=d(x,z)$, $\delta_1>0$ given by Lemma \ref{lemaWillyBarney} for $\eps_1$ and choose $k_0\in\N$ such that $$d(f^k(z),f^k(x))<\delta_1\,\,\,\,\,\, \text{whenever} \,\,\,\,\,\, |k|\geq k_0.$$ Then Lemma \ref{lemLshByAsint} assures the existence of a compact set $K\subset X$ such that $$\sup_{k\in\Z}\diam(f^k(K))\leq 2\eps$$ and $N\geq 1$ such that $f^N(K)=K$ and
$f^N\colon K\to K$ is semi-conjugate to a shift. Since the orbit of $K$ is contained in the $2\eps$-neighborhood of $\Lambda$, it is indeed contained in $\Lambda$ by the choice of $r$. Since this can be done for each $\eps\in(0,r)$, it follows that $\Lambda$ admits arbitrarily small topological semi-horseshoes.
\end{proof}

\begin{proof}
[Proof of Corollary \ref{CountableEntropy}]
In general, expansivity implies
countable-expansivity and
entropy-expansivity.
As topological semi-horseshoes have positive entropy and contain
uncountably many points, the results follows from Theorem \ref{teoEspectralLsh}.
\end{proof}

\section{Examples}

In this Section we prove Theorems \ref{Examples} and \ref{aperiodic}, where examples of homeomorphisms satisfying the L-shadowing property are obtained.

\vspace{+0.3cm}

\subsection{Structurally stable diffeomorphisms}

We begin by proving the L-shadowing property for structurally stable diffeomorphisms. This is a corollary of the following more general result, since structurally stable diffeomorphisms admit the shadowing property (see \cite{Rob}) and have an expansive non-wandering set (see \cite{Man}).

\begin{theorem}\label{non-wandering}
If $f$ is a homeomorphism, defined in a compact metric space, satisfying the shadowing property and such that $\Omega(f)$ is expansive, then $f$ has the L-shadowing property.
\end{theorem}

To prove this theorem
we use the following proposition, which is well known for $f\colon X\to X$ expansive. 
We will only assume expansivity for the restriction to the non-wandering set.

\begin{prop}\label{propOmegaExp}
If $\Omega(f)$ is expansive, with expansivity constant $c>0$,
then $$W^s_c(x)\subset W^s(x) \,\,\,\,\,\, \text{and} \,\,\,\,\,\, W^u_c(x)\subset W^u(x) \,\,\,\,\,\, \text{for all} \,\,\,\,\,\, x\in X.$$
\end{prop}

\begin{proof}
Suppose that $\dist(f^n(x),f^n(y))\leq c$ for all $n\geq 0$.
Arguing by contradiction, suppose that there are $n_k\to+\infty$ and $\epsilon>0$ such that
$$\dist(f^{n_k}(x),f^{n_k}(y))>\epsilon \,\,\,\,\,\, \text{for all} \,\,\,\,\,\,k\in\N.$$
As $X$ is compact, we can assume that
$f^{n_k}(x)\to x_*$ and $f^{n_k}(y)\to y_*$, when $k\to\infty$, where $x_*,y_*\in\Omega(f)$.
From the continuity of $f$ we have
$$\epsilon\leq\dist(f^i(x_*),f^i(y_*))\leq c\,\,\,\,\,\, \text{for all} \,\,\,\,\,\,i\in\Z.$$
This contradicts that $c$ is an expansivity constant of $\Omega(f)$ and ends the proof.
\end{proof}

\begin{proof}[Proof of Theorem \ref{non-wandering}]
Let $c>0$ be such that $\Gamma_c(x)=\{x\}$ for every $x\in\Omega(f)$. For each $\eps\in(0,c)$, the shadowing property assures the existence of $\delta\in(0,\eps)$ such that $$W^u_{\eps}(x)\cap W^s_{\eps}(y)\neq\emptyset \,\,\,\,\,\, \text{whenever} \,\,\,\,\,\, d(x,y)<\delta.$$  The previous proposition assures that $$W^s_{\eps}(x)\subset W^s(x) \,\,\,\,\,\, \text{and} \,\,\,\,\,\, W^u_{\eps}(x)\subset W^u(x) \,\,\,\,\,\, \text{for all} \,\,\,\,\,\, x\in X.$$ 
Then, it follows that $V^s_{\eps}(x)\cap V^u_{\eps}(x)\neq\emptyset$ whenever $d(x,y)<\delta$. Since this can be done for each $\eps>0$, Theorem \ref{teoCharLsh} proves the L-shadowing property.
\end{proof}

\vspace{+0.3cm}

\subsection{Pseudo-Anosov diffeomorphisms of $\mathbb{S}^2$}

Pseudo-Anosov diffeomorphisms of the sphere $\mathbb{S}^2$ can be constructed as follows. Consider an Anosov diffeomorphism $f_A$ of the torus $\mathbb{T}^2$ induced by a hyperbolic $2\times2$ matrix with integer coefficients and determinant one. The sphere $\mathbb{S}^2$ can be seen as the quotient of $\mathbb{T}^2$ by the antipodal map and then $f_A$ induces a homeomorphism $g_A\colon \mathbb{S}^2\to\mathbb{S}^2$. To prove that $g_A$ has the L-shadowing property, it is enough to prove the following theorem, since the antipodal quotient is an open map:

\begin{theorem}
\label{teoProjLsh}
Let $M$ and $N$ be compact metric spaces, $f\colon M\to M$ and $g\colon N\to N$ be homeomorphisms and $q\colon M\to N$ a continuous, onto and open map such that $q\circ f=g\circ q$. If $f$ has the $L$-shadowing property, then $g$ also has the $L$-shadowing property.
\end{theorem}

\begin{proof}
The argument is based on \cite{AAV}*{Proposition 5.2} where the shadowing property of $g$ is proved assuming the shadowing property of $f$. Now we suppose that $f$ has the $L$-shadowing property and prove the $L$-shadowing property for $g$. Let $\eps>0$ be arbitrary and $\eps'>0$ be given by the uniform continuity of $q$ such that $$d(a,b)<\eps' \,\,\,\,\,\, \text{implies} \,\,\,\,\,\, d(q(a),q(b))<\eps.$$ Choose $\delta'>0$, given by the L-shadowing property of $f$ such that every $(\delta',L)$-pseudo orbit of $f$ is $(\eps',L)$-shadowed. Since $q$ is an open map, \cite{AAV}*{Lemma 5.1} assures the existence of $\delta>0$ such that $$B(q(x),\delta)\subset q(B(x,\delta')) \,\,\,\,\,\, \text{for every} \,\,\,\,\,\, x\in X.$$
Let $(x_k)_{k\in\Z}$ be a $(\delta,L)$-pseudo orbit of $g$. We will lift it to a $(\delta',L)$-pseudo orbit $(y_k)_{k\in\Z}$ in $M$ satisfying $$y_k\in q^{-1}(x_k) \,\,\,\,\,\, \text{for every} \,\,\,\,\,\, k\in\Z.$$ We first lift the positive part of $(y_k)_{k\in\Z}$ and with a similar argument we lift its negative part. For each $j\in\N$, let $\eps_j=\frac{\delta'}{j+1}$ and choose $\delta_j>0$ such that $$B(q(x),\delta_j)\subset q(B(x,\eps_j)) \,\,\,\,\,\, \text{for every} \,\,\,\,\,\, x\in X \,\,\,\,\,\, \text{and} \,\,\,\,\,\, j\in\N.$$
Since $(x_k)_{k\in\N}$ is a limit pseudo orbit of $g$, we can choose an increasing sequence $(k_j)_{j\in\N}$ of natural numbers such that $$d(g(x_k),x_{k+1})<\delta_j \,\,\,\,\,\, \text{for every} \,\,\,\,\,\, k\geq k_j.$$ We will define $(y_k)_{k\in\N}$ by induction in each interval of natural numbers between $k_j$ and $k_{j+1}$ in such a way that $$d(f(y_k),y_{k+1})<\eps_j \,\,\,\,\,\, \text{whenever} \,\,\,\,\,\, k_j< k\leq k_{j+1}.$$
Let $y_0$ be any point in $q^{-1}(x_0)$ and note that $d(g(x_0),x_1)<\delta$ implies the existence of $y_1\in B(f(y_0),\delta')$ such that $q(y_1)=x_1$. Also, $d(g(x_1),x_2)<\delta$ implies the existence of $y_2\in B(f(y_1),\delta')$ such that $q(y_2)=x_2$. Repeating this argument we define $y_k$ for $0\leq k\leq k_1$. Since $d(g(x_{k_1}),x_{k_1+1})<\delta_1$, there exists $y_{k_1+1}\in B(f(y_{k_1}),\eps_1)$ such that $q(y_{k_1+1})=x_{k_1+1}$. We repeat this argument again to define $y_k$ for $k_1< k\leq k_2$. An induction process defines the sequence $(y_k)_{k\in\N}$ with desired properties since
$$d(g(x_k),x_{k+1})<\delta_j \,\,\,\,\,\, \text{whenever} \,\,\,\,\,\, k_j<k\leq k_{j+1}$$ and, hence, $d(f(y_k),y_{k+1})<\eps_j$. It follows that $(y_k)_{k\in\N}$ is a $\delta'$-pseudo orbit of $f$, since $\eps_j<\delta'$ for every $j\in\N$, and also a limit pseudo orbit of $f$ since $\eps_j\to0$ when $j\to\infty$. With a similar argument we lift the negative part of $(x_k)_{k\in\Z}$ and define the whole sequence $(y_k)_{k\in\Z}$.
Finally, the $L$-shadowing property of $f$ assures the existence of $z\in X$ that $(\eps',L)$-shadows $(y_k)_{k\in\Z}$ and, hence, $q(z)$ $(\eps,L)$-shadows $(x_k)_{k\in\Z}$. This proves the $L$-shadowing property for $g$.
\end{proof}

A corollary of this theorem and Proposition \ref{mixing} is the following.

\begin{corollary}
The pseudo-Anosov diffeomorphism $g_A$ of the two-dimensional sphere satisfies the two-sided limit shadowing property.
\end{corollary}

\begin{proof}
Note that $g_A$ is topologically mixing since $f_A$ is topologically mixing, and satisfies the L-shadowing property by Theorem \ref{teoProjLsh}, so it has the two-sided limit shadowing property by Proposition \ref{mixing}.
\end{proof}

\subsection{L-shadowing without periodic points}

The following example was considered in \cite{CK} as an example of a homeomorphism with the two-sided limit shadowing property but without periodic points. We will prove that it satisfies the L-shadowing property and exhibit arbitrarily small topological semi-horseshoes in it.

\begin{proof}[Proof of Theorem \ref{aperiodic}]For each $r>1$ consider the set $\{0,\dots,r-1\}$ endowed with the discrete metric $\rho$, let $\Omega_r=\{0,\dots,r-1\}^{\Z}$ and consider in $\Omega_r$ the Tychonoff product topology. Let $p$ and $q$ be relatively prime integers and $X_{(p,q)}$ be the set of all sequences in $\Omega_{p+q-1}$ whose entries are vertices visited during a bi-infinite walk on the directed graph with two loops, one of length $p$ and one of length $q$ as shown in Figure \ref{figGraph-pq}. 
See \cite{CK} for more details.
\begin{center}
\begin{figure}[ht]\label{fig}
\includegraphics{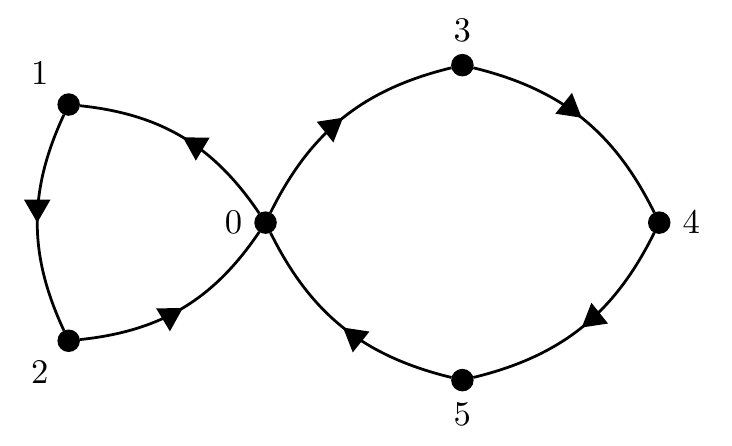}
\caption{A graph presenting the shift space $X_{(3,4)}$.}
\label{figGraph-pq}
\end{figure}
\end{center}
It is clear that $X_{(p,q)}$ is invariant by the shift map on $\Omega_{p+q-1}$ and is a subshift of finite type. Moreover, it does not have any periodic point with period smaller than $\min\{p,q\}$. Let $(p_n)_{n=1}^\infty$ be a strictly increasing sequence of prime numbers. For each $n\in\N$ let $X_n=X_{(p_n,p_{n+1})}$ and $\sigma_n$ be the shift transformation on $\Omega_{p_{n}+p_{n+1}-1}$ restricted to $X_n$. Consider the product system $F=\sigma_1\times\sigma_2\times\ldots$ on $X=\prod_{n=1}^\infty X_n$ and note that $F$ has no periodic points since the coordinates $x_n$ of a periodic point of $F$ would be periodic points of the respective shift map $\sigma_n$, that only admits periodic points with period at least $p_n$, so this would contradict the fact that $p_n\to\infty$ when $n\to\infty$. Also, $F$ is topologically mixing as a consequence of Theorem 5.3 in \cite{CK}.


Now we prove that $F$ has the $L$-shadowing property. First, we clarify some notation. If $a\in X$, then $a=(a_n)_{n\in\N}$ where $a_n\in X_n$ 
for every $n\in\N$, and each $a_n$ is equal to a sequence denoted by $(a_{n,k})_{k\in\Z}$. 
In this way, $a_{n,k}$ denotes the element of position $k$ in the sequence $a_n$, that, in turn, 
is the element of position $n$ of $a$. The metric $d$ of $X$ is defined as follows: if $x,y\in X$ then
$$d(x,y)=\sum_{n=1}^{+\infty}\frac{d_n(x_n,y_n)}{2^n}$$
where $d_n$ is the metric in $X_n$ defined by $$d_n(x_n,y_n)=\sum_{k\in\Z}\frac{\rho(x_{n,k},y_{n,k})}{2^{|k|}}$$ and $\rho$ denotes the discrete metric. 

It is proved in \cite{CK} that $F$ has the two-sided limit shadowing property and, hence, the shadowing property.\footnote{Notice that a positive answer to 
Question \ref{tslsandLs} would finish the proof.} 
For each $\eps>0$, consider $\delta>0$, given by the shadowing property of $F$, such that every $\delta$-pseudo orbit of $F$ is $\frac{\epsilon}{2}$-shadowed. Let $(l^n)_{n\in\N}\subset X$ be a $(\delta,L)$-pseudo orbit of $F$ and consider two points $r$ and $s$ in $X$ such that $(l^n)_{n\in\N}$ is two-sided limit shadowed by $r$ and $\frac{\eps}{2}$-shadowed by $s$. 
To obtain a point $(\eps,L)$-shadowing $(l^n)_{n\in\N}$ we will define a new point $w\in X$ mixing the coordinates of $r$ and $s$, so that the coordinates of $s$ that are not near enough the pseudo orbit are changed to the coordinates of $r$ to be close enough.

We define $w\in X$ as follows: for each $n\in\mathbb{N}$, let $w_n=s_n$ if there is $n_0>0$ such that 
$$s_{n,k}=r_{n,k} \,\,\,\,\,\, \text{whenever} \,\,\,\,\,\, |k|>n_0;$$ otherwise, 
let $w_n=r_n$. We claim that $w$ defined this way $(2\eps,L)$-shadows $(l^n)_{n\in\N}$.
Let $N$ be the first positive integer number satisfying $w_N\neq s_N$ (if such number does not exist, 
then $s$ is easily seen to $(\eps,L)$-shadow $(l^n)_{n\in\N}$). 
Since $(l^n)_{n\in\N}$ is two-sided limit shadowed by $r$ and $\frac{\eps}{2}$-shadowed by $s$, we can choose $J>0$ such that $$d(F^k(r),F^k(s))\leq\eps\,\,\,\,\,\, \text{whenever} \,\,\,\,\,\, |k|>J.$$ Since $w_N\neq s_N$, 
there exists $j\in\Z$ such that $|j|>J$ and $r_{N,j}=w_{N,j}\neq s_{N,j}$. This implies that 
$$d_N(\sigma_N^j(r_N),\sigma_N^j(s_N))>1$$ since $r_{N,j}$ and $s_{N,j}$ are the elements of order 0 in 
$\sigma_N^j(r_N)$ and $\sigma_N^j(s_N)$, respectively. Recall that $$d(F^j(r),F^j(s))=\sum_{n=1}^{+\infty}\frac{d_n(\sigma_n^j(r_n),\sigma_n^j(s_n))}{2^n},$$ so that $d(F^j(r),F^j(s))\leq\eps$ implies that each term of this sum is smaller than $\eps$. In particular, the term of order $N$ satisfies $$\frac{1}{2^N}<\frac{d_N(\sigma_N^j(r_N),\sigma_N^j(s_N))}{2^N}<\eps.$$
Thus, for each $a,b\in X$ it follows that $$\sum_{n=N}^{+\infty}\frac{d_n(a_n,b_n)}{2^n}=\frac{1}{2^N}\sum_{n=1}^{+\infty}\frac{d_{n+N}(a_{n+N},b_{n+N})}{2^n}<\eps$$ and for each $i\in\Z$ we have
\begin{eqnarray*} d(F^i(w),l_i)&=&\sum_{n=1}^{N-1}\frac{d_n(\sigma^i(w_n),l^i_n)}{2^n}+\sum_{n=N}^{+\infty}\frac{d_n(\sigma^i(W_n),l^i_n)}{2^n} \\
&\leq&\sum_{n=1}^{N-1}\frac{d_n(\sigma^i(s_n),l^i_n)}{2^n}+\eps\leq 2\eps.
\end{eqnarray*}
This proves that  $(l^n)_{n\in\N}$ is $2\eps$-shadowed by $w$. 

To see it is also two-sided limit shadowed by $w$, consider for each $\gamma>0$ a number $M>0$ satisfying $$\sum_{n=M+1}^{+\infty}\frac{1}{2^n}<\frac{\gamma}{6}$$ and choose $k\in\N$ such that $w_{n,i}=r_{n,i}$ whenever $|i|>k$ and $n\leq M$. Since $r$ two-sided limit shadows $(l^n)_{n\in\N}$ one can find $P>0$ such that 
$$\sum_{n=1}^M\frac{d_n(\sigma^j(w_n),l^j_n)}{2^n}<\frac{\gamma}{2} \,\,\,\,\,\, \text{whenever} \,\,\,\,\,\, |j|>P.$$ Thus, if $|j|>P$, then
$$d(F^j(w),L^j)=\sum_{n=1}^ M\frac{d_n(\sigma^j(w_n),l^j_n)}{2^n}+\sum_{n=M+1}^{+\infty}\frac{d_n(\sigma^j(w_n),l^j_n)}{2^n}<\gamma.$$
Since this can be done for every $\gamma>0$, it follows that $w$ two-sided limit shadows $(l^n)_{n\in\N}$ and the L-shadowing property is proved.
\end{proof}

\begin{remark}\label{rmkSemiFerrExplicitas}
Let us construct arbitrarily small topological semi-horseshoes in the example given in the previous proof. We continue using the same notation. 
For each $n\in\N$ we will choose $Q_n\subset X_n$ and $k_n\in\N$ such that $\sigma_n^{k_n}(Q_n)=Q_n$ and $\sigma_n^{k_n}\colon Q_n\to Q_n$ is semi-conjugate to a shift of two symbols $\sigma_2$ in $\{a,b\}^{\Z}$. Let $Q_n$ be the set of sequences in $X_n$ where each loop with length $p_n$ appears in the sequence in blocks repeated $p_{n+1}$ times and each loop with length $p_{n+1}$ appears in the sequence in blocks repeated $p_n$ times. Let $k_n=p_np_{n+1}$ and note that $\sigma_n^{k_n}(Q_n)=Q_n$.
Define $h_n\colon Q_n\to \{a,b\}^{\Z}$ as follows: for each sequence $y=(y_k)_{k\in\Z}\in Q_n$ we associate to each block containing $p_{n+1}$ copies of the loop of length $p_n$ the symbol $a$ and to each block containing $p_n$ copies of the loop of length $p_{n+1}$ the symbol $b$, defining the element of order zero in $h_n(y)$ to be the symbol associated to the loop containing the elements $y_0$ and $y_1$ and the other coordinates of $h_n(y)$ so that the previous association preserves the order of the blocks in $y$. Then $h_n$ is clearly surjective, continuous and satisfies $h_n\circ\sigma_n^{k_n}=\sigma_2\circ h_n$.

For each $n\in\N$ and $i\in\{1,\dots,n-1\}$ consider $q_i$ a periodic point of $\sigma_i$ in $X_i$ and define $$K_n=\{(q_1,q_2,\dots,q_{n-1},x,y_{n+1},y_{n+2},\dots);\,\,\, x\in Q_n \,\,\, \text{and} \,\,\, y_j\in X_j \,\,\, \text{for each} \,\,\, j>n\}.$$ For each $\eps>0$ choose $n\in\N$ such that
$$\sum_{i=n}^{+\infty}\frac{1}{2^{i+1}}<\eps$$ and consider $Q_n, k_n$ and $K_n$ as above. Let $\pi(q_i)$ denote the period of $q_i$ and consider
$$N=\pi(q_1)\pi(q_2)\dots\pi(q_{n-1})k_n.$$ Note that each $q_i$ is a fixed point of $\sigma_i^N$, that $\sigma_n^{N_n}(Q_n)=Q_n$ and that $X_j$ is invariant by $\sigma_j$ for each $j>n$. This imply that $F^N(K_n)=K_n$.

Consider $h_n:Q_n\to \{a,b\}^{\Z}$ the semi-conjugacy map between $\sigma_n^{k_n}$ and $\sigma_2$ and let $\pi_n\colon X\to X_n$ denote the projection of $X$ onto $X_n$. If $h\colon K_n\to\{a,b\}^{\Z}$ is defined by $h=h_n\circ\pi_n$ and $M=N/k_n$, then $h\circ F^N=\sigma_2^{M}\circ h$ and $h$ is a semi-conjugacy map between $F^N$ restricted to $K_n$ and $\sigma_2^M$. Also note that $$\sup_{k\in\Z}\diam(F^k(K_n))\leq\eps$$ since points in $F^k(K_n)$ have the same $n-1$ first coordinates and the other coordinates are bounded by $\frac{1}{2^{i+1}}$. Since this can be done for every $\eps>0$, we obtained arbitrarily small topological semi-horseshoes for $F$.

\end{remark}







\section{Positive expansivity}

In this section, we obtain Theorem \ref{h} as a direct consequence of the next result, since the shadowing property easily implies the first condition, while positive finite-expansivity implies the second.

\begin{prop}
 Let $f$ be a homeomorphism of a compact metric space satisfying:
 \begin{enumerate}
  \item for any $\epsilon>0$ there is $\delta>0$ such that $\dist(a,b)<\delta$ implies
  $W^s_\epsilon(a)\cap W^u_\epsilon(b)\neq\emptyset$,
  \item for each $x\in X$ there is $c>0$ such that $W^s_c(x)=\{x\}$.
 \end{enumerate}
This implies that $X$ is a finite set.
\end{prop}

\begin{proof}
Given $x\in X$ consider $c>0$, given by (2), such that $W^s_c(x)=\{x\}$.
Choose $\delta>0$ such that
$$\dist(a,b)<\delta\,\,\,\,\,\, \text{implies} \,\,\,\,\,\, W^s_c(a)\cap W^u_c(b)\neq\emptyset.$$
If $y\in B(x,\delta)$,
then (1) assures that $W^s_c(x)\cap W^u_c(y)\neq\emptyset$. Since $W^s_c(x)=\{x\}$, we conclude that $x\in W^u_c(y)$.
This implies that $y\in W^u_c(x)$ and, hence, $$B(x,\delta)\subset W^u_c(x).$$
This implies that $f^{-1}$ is equicontinuous (equivalently, each point is Lyapunov stable for $f^{-1}$).
By \cite{AG}*{Theorem 3.4} we have that $f$ is equicontinuous.
This implies that each point of $X$ is isolated, and as $X$ is compact we conclude that $X$ is a finite set.
\end{proof}

\section*{Acknowledgements}
The second author was supported by Capes, CNPq and the Alexander von Humboldt Foundation. Part of this work was developed while the second author was visiting the Departamento de Matem\'atica y Estad\'\i stica del Litoral in Salto, Uruguay, where some conversations with Mauricio Achigar happened.

\vspace{1.5cm}
\noindent

{\em A. Artigue and J. Vieitez}
\vspace{0.2cm}

\noindent

Departamento de Matem\'atica y Estadística del Litoral,

Universidad de la República,

Gral. Rivera 1350, Salto, Uruguay
\vspace{0.2cm}

\email{artigue@unorte.edu.uy}

\email{jvieitez@unorte.edu.uy}

\vspace{1.5cm}
\noindent

{\em B. Carvalho}
\vspace{0.2cm}

\noindent

Departamento de Matem\'atica,

Universidade Federal de Minas Gerais - UFMG

Av. Ant\^onio Carlos, 6627 - Campus Pampulha

Belo Horizonte - MG, Brazil.

\vspace{0.2cm}
Friedrich-Schiller-Universität Jena

Fakultät für Mathematik und Informatik

Ernst-Abbe-Platz 2

07743 Jena

\vspace{0.2cm}

\email{bmcarvalho@mat.ufmg.br}

\vspace{1.5cm}
\noindent

{\em W. Cordeiro}

\noindent

Institute of Mathematics, Polish Academy of Sciences

ul. \'Sniadeckich, 8

00-656 Warszawa - Poland

\vspace{0.2cm}

\email{wcordeiro@impan.pl}

\end{document}